\documentclass{aims}
\usepackage{amsmath}
\usepackage{amssymb}
  \usepackage{paralist}
  \usepackage{graphics} 
  \usepackage{epsfig} 
\usepackage{graphicx}  \usepackage{epstopdf}
 \usepackage[colorlinks=true]{hyperref}
\hypersetup{urlcolor=blue, citecolor=red}

\def\currentvolume{X}
 \def\currentissue{X}
  \def\currentyear{200X}
   \def\currentmonth{XX}
    \def\ppages{X--XX}
     \def\DOI{10.3934/xx.xx.xx.xx}

\newtheorem{theorem}{Theorem}[section]
\newtheorem{corollary}{Corollary}

\newtheorem{proposition}{Proposition}

\theoremstyle{definition}
\newtheorem{definition}[theorem]{Definition}
\newtheorem{remark}{Remark}

\title[On motions without falling of an inverted pendulum] 
      {On motions without falling of an inverted pendulum with dry friction}

\author[Ivan Polekhin]{}


 \email{ivanpolekhin@mi.ras.ru}

\thanks{This work was supported by the Program of the Presidium of the Russian Academy of Sciences No 01 ‘Fundamental Mathematics and its
Applications’ under grant PRAS-18-01.}

\thanks{$^*$ Corresponding author: Ivan Polekhin}

\begin{document}
\maketitle

\centerline{\scshape Ivan Polekhin$^*$}
\medskip
{\footnotesize
 \centerline{Steklov Mathematical Institute of the Russian Academy of Sciences}
   \centerline{ Moscow, Russia}
} 



\bigskip

 \centerline{(Communicated by ...)}

\begin{abstract}
An inverted planar pendulum with horizontally moving pivot point is considered. It is assumed that the law of motion of the pivot point is given and the pendulum is moving in the presence of dry friction. Sufficient conditions for the existence of solutions along which the pendulum never falls below the horizontal positions are presented. The proof is based on the fact that solutions of the corresponding differential inclusion are right-unique and continuously depend on initial conditions, which is also shown in the paper. 
\end{abstract}

\section{Mechanical model}
Consider an inverted planar pendulum in a gravitational field with its pivot point moving along a horizontal line in the plane of the pendulum. The law of motion $\xi(t)$ of the pivot point is given and the pendulum is moving in the presence of dry friction. Let $l$ be the length of the pendulum, $m$ be its mass. Let $r = (r_x, r_y)$ be the radius vector of the massive point of the pendulum, and $r_x$, $r_y$ are its components in axes of an orthogonal coordinate system $Oxy$, where $Oy$ is the vertical axis. The general equation of motion has the form
$$
m \ddot r = F_{grav} + F_{fric} + N.
$$
Here $F_{grav} = -mg\cdot e_y$ is the applied force of gravity, $F_{fric}$ is the force of dry friction, and $N$ is the force of constraint that appears from the holonomic constraint $(r_x - \xi(t))^2 + r_y^2 = l^2$. By $g$ we denote the gravitational acceleration. We assume that the force of dry friction is the Coulomb friction. In our model we consider the case when the Stribeck effect can be ignored, and the difference between the dynamic and static friction coefficients is negligibly small. Therefore, $F_{fric}$ is as follows \cite{popov2010, ivanov2009bifurcations}
$$
F_{fric} = -\mu |N| \frac{v}{|v|}, \mbox{ if } v \ne 0, \quad |F_{fric}| \leqslant \mu |N|,
\mbox{ if } v = 0.
$$
Here $\mu > 0$ is the dry friction coefficient, $N$ is the normal reaction force, and $v$ is the relative velocity of the massive point: $v = (\dot r_x - \dot \xi)e_x + \dot r_y e_y$. 

 Let $q$ be the angle between the horizontal line and the rod ($q = 0$ or $q = \pi$ are the horizontal positions). It is not hard to obtain that $|N| = m |\ddot \xi \cos q -l\dot q^2 + g\sin q|$. Therefore, for $v \ne 0$, the equation of motion can be presented as follows
 \begin{equation}
 \label{eq1}
\begin{aligned}
    &\dot q = p,\\
    &\dot p = \frac{\ddot\xi}{l}\sin q -\frac{\mu}{l}\left| \ddot\xi \cos q - lp^2 + g\sin q \right|\frac{p}{|p|} - \frac{g}{l}\cos q.
\end{aligned}
\end{equation}
When $p = 0$, $|F_{fric}|$ can be any value between zero and $\mu |N|$. Therefore, the motion of the system cannot be described by an ordinary differential equation. One of the possible and convenient solution to this problem is to consider a differential inclusion corresponding to the considered model of dry friction, which we do in the next section.

The system of a one-dimensional inverted pendulum --- being a simple yet strongly non-linear system --- has been considered by many authors (see, for instance, \cite{kapitsa1954pendulum, bardin1995stability, butikov2001dynamic, seyranian2006stability}). Many of these results deal with the smooth system of a pendulum with vertically oscillating pivot. Unlike these cases, we consider a non-smooth system with dry friction and the pivot point moving horizontally and show that there always exists a solution along which the inverted pendulum never falls below the horizontal line.

\section{Main result}
The system (\ref{eq1}) and similar systems can be presented in the following form
\begin{equation}
\label{eq2}
\dot x = f(x, t),
\end{equation}
where $f$ is a piecewise continuous function in a domain $G \subset \mathbb{R}^{n+1}$ and $M \subset G$ is a set of measure zero of points of discontinuity of $f$. Following \cite{filippov2013differential}, consider a differential inclusion associated with the above equation (\ref{eq2})
\begin{equation}
\label{eq3}
\dot x \in F(x,t),
\end{equation}
where $F \colon G \to 2^{\mathbb{R}^{n}}$ is a set-valued function defined as follows: for any point $(x,t) \in G$, the set $F(x,t)$ is the smallest convex closed set containing all the limit values of $f(x^*,t)$, $(x^*,t) \notin M$, $x^* \to x$.
\begin{definition}
A solution of the differential inclusion (\ref{eq3}) is an absolutely continuous function $x \colon I \to \mathbb{R}^n$ defined on an interval or on a segment $I$ for which (\ref{eq3}) is satisfied almost everywhere. 
\end{definition}
Below, by the solution of (\ref{eq1}) we mean the solution of the corresponding differential inclusion with the right-hand side denoted by $\Phi = \Phi(q,p,t)$ (in our case, $M$ is the plane $p = 0$). We also assume that $\ddot \xi $ in (\ref{eq1}) is a Lipschitz function. 

First, we show that the existence of solutions and their continuous dependence on initial data follow directly from the general properties of $\Phi$.
\begin{definition}
Let $A, B \subset \mathbb{R}^n $ be non-empty closed sets. Then
$$
\beta(A,B) = \sup\limits_{a \in A}\rho(a, B).
$$
Here $\rho(a, B) = \inf\limits_{b \in B} \rho(a,b)$ and $\rho(a,b)$ is the Euclidean distance in $\mathbb{R}^n$.
\end{definition}
\begin{definition}
A set-valued function $F$ is called upper semicontinuous at $x \in \mathbb{R}^n$, if $\beta(F(y),F(x)) \to 0$ as $y \to x$. A function is called upper semicontinuous on a set $G$ if it is upper semicontinuous at each point of $G$. 
\end{definition}

It is not hard to see that $\Phi$ is upper semicontinuous in $\mathbb{R} / 2\pi\mathbb{Z} \times \mathbb{R} \times \mathbb{R}$.
\begin{theorem}\cite{filippov2013differential}
Let $F$ be an upper semicontinuous set-valued function in a domain $G \subset \mathbb{R}^{n+1}$, and for all $(x, t) \in G$, $F(x, t)$ is a non-empty, bounded, closed and convex set. Then for any point $(x_0, t_0) \in G$ there exists a local solution of the problem
$$
\dot x \in F(x,t), \quad x(t_0) = x_0.
$$
Moreover, if $G$ is closed and bounded, then every solution can be continued up to the boundary of $G$.
\end{theorem}

From this theorem, it follows that for the system (\ref{eq1}), solutions exist for all $t > t_0$. Indeed, set-valued function  $\Phi$ is periodic in $q$ and it can be shown that if $p > 0$ and
$$
p^2 > p_*^2 = \left( g + \max\limits_{t \in [t_0,t_1]}|\ddot \xi| \right) \left(1 + \frac{1}{\mu} \right)\frac{1}{l},
$$
then $\dot p < 0$ for any $t \in [t_0, t_1]$. Similarly, for $p<0$ and large $|p|$, we have $\dot p > 0$. Therefore, solutions can be continued to an arbitrarily long time interval. From the below result, it also follows that solutions of (1) depend continuously on initial data.
\begin{theorem} \cite{filippov2013differential}
Let $F$ be an upper semicontinuous set-valued function in a domain $G \subset \mathbb{R}^{n+1}$, and for all $(x, t) \in G$, $F(x, t)$ is a non-empty, bounded, closed and convex set; $t_0 \in [a, b]$, let all the solutions of the problem
$$
\dot x \in F(x,t), \quad x(t_0) = x_0
$$
exist for $a \leqslant t \leqslant b$ and their graphs lie in $G$. Then for any $\varepsilon > 0$ there exists $\delta > 0$, such that for any $t_0^* \in [a,b]$ and $x_0^*$, $|t_0^* - t_0|<\delta$ and $|x_0^* - x_0|<\delta$, each solution with initial conditions $t = t_0^*$, $x = x_0^*$ exists and differs (w.r.t. the uniform norm) from some solution with initial conditions $t = t_0$, $x = x_0$ by not more than $\varepsilon$.
\end{theorem}

\begin{definition}
We say that (\ref{eq3}) has a right-unique solution at a point $(x_0, t_0)$ if there exists $t_1 > t_0$ such that each two solutions of the differential inclusion satisfying the condition $x(t_0) = x_0$ coincide on $[t_0, t_1]$.
\end{definition}

Let us now show that, for given initial conditions, the solution of (\ref{eq2}) is right-unique. The following result was also proved in \cite{filippov2013differential}
\begin{theorem}
Let a function $f(t,x)$ in a domain $G$ be discontinuous only on a set of measure zero. Let there exist a summable function $l(t)$ such that for almost all points $(x,t)$ and $(y,t)$ of the domain $G$ we have $f(x,t) \leqslant l(t)$ and for $|x - y| < \varepsilon_0$,  $\varepsilon_0 > 0$, the following holds
\begin{equation}
\label{eq4}
(x-y)\cdot(f(x,t) - f(y,t)) \leqslant l(t) |x-y|^2.
\end{equation}
Then any solution of the corresponding differential inclusion (\ref{eq3}) is right-unique in the domain $G$.
\end{theorem}
\begin{remark}
As usual, we say that function $l \colon \mathbb{R} \to \mathbb{R}$ is summable if it is Lebesgue integrable and
$$
\int\limits_K|l(t)|\, dt < \infty,
$$
For any compact $K$. Below we consider only constant functions $l(t) = l$ that are always summable.
\end{remark}
\begin{corollary}
The solutions of (\ref{eq1}) are right-unique.
\end{corollary}
\begin{proof}
Let $G$ be a bounded domain in $\mathbb{R}^{3}$, by $G^+$ we denote $\{ p > 0 \} \cap G$. Similarly,  $G^- = \{ p < 0 \} \cap G$. By $f(q,p,t)$ we denote the right-hand side of the system (1). Let $f^+(q,0,t)$ and $f^-(q,0,t)$ be the limiting values of the function $f$ at the point $(q,0,t)$, from $G^+$ and $G^-$, correspondingly. Let $n$ be a vector directed toward increasing values of $p$. From (\ref{eq1}) we have
$$
n \cdot f^+(q,0,t) = \frac{\ddot\xi}{l}\sin q -\frac{\mu}{l}\left| \ddot\xi \cos q - lp^2 + g\sin q \right| - \frac{g}{l}\cos q,
$$
and
$$
n \cdot f^-(q,0,t) = \frac{\ddot\xi}{l}\sin q + \frac{\mu}{l}\left| \ddot\xi \cos q - lp^2 + g\sin q \right| - \frac{g}{l}\cos q.
$$
Therefore, $n \cdot f^+(q,0,t) \leqslant n \cdot f^-(q,0,t)$. Let us now show that for almost all points $(q_1,p_1,t)$ and $(q_2, p_2,t)$, inequality (4) holds.

If both points are in $G^+$ or in $G^-$, then the inequality follows from the fact that the right-hand side is Lipschitz continuous (in this case, we can put $l(t)$ to be a constant). Now suppose that $(q_1,p_1,t) \in G^+$ and $(q_2, p_2,t) \in G^-$. By $(q, 0, t)$ we denote the point of intersection of the line segment connecting $(q_1,p_1,t)$ and $(q_2, p_2,t)$ with the plane $p = 0$. Since $f$ is Lipschitz continuous in $G^-$ and $G^+$, then for some constant $l$, we have the following inequalities.
\begin{equation*}
\begin{aligned}
&|f(q_1, p_1, t) - f^+(q, 0, t)| \leqslant l |(q_1, p_1,t) - (q, 0,t)|,\\
&|f^-(q, 0, t) - f(q_2, p_2, t)| \leqslant l |(q_2, p_2,t) - (q, 0,t)|.
\end{aligned}
\end{equation*}
From these inequalities and the fact that the points $(q_1,p_1,t)$, $(q,0,t)$, $(q_2,p_2,t)$ are on the same line, we have
\begin{equation*}
\begin{aligned}
|f(q_1, p_1, t) - f^+(q, 0, t) + f^-(q, 0, t) - f(q_2, p_2, t)| \leqslant l |(q_1, p_1,t) - (q_2, p_2,t)|.
\end{aligned}
\end{equation*}
Therefore,
\begin{equation*}
\begin{aligned}
((q_1, p_1,t) - (q_2, p_2,t))&\cdot (f(q_1, p_1, t) - f^+(q, 0, t) + f^-(q, 0, t) - f(q_2, p_2, t)) \\&\leqslant l |(q_1, p_1,t) - (q_2, p_2,t)|^2.
\end{aligned}
\end{equation*}
Note that $f^+(q, 0, t) - f^-(q, 0, t)$ is parallel to $n$ or equals zero, therefore
$$
((q_1, p_1,t) - (q_2, p_2,t)) \cdot (f^+(q, 0, t) - f^-(q, 0, t)) \leqslant 0.
$$
Finally, if we sum the last two inequalities, we obtain (\ref{eq4}). When $G$ is not a bounded region, it can be presented as a union of bounded sets, in which the solutions are right-unique.
\end{proof}
\begin{remark}
We note, that the presented proof is similar to the one in \cite{filippov2013differential}, yet it covers a wider class of functions (in \cite{filippov2013differential}, $f$ is assumed to be twice-differentiable almost everywhere).
\end{remark}

We have shown that any solution of (\ref{eq1}) exists for all $t \geqslant t_0$, this solution is right-unique and continuously depends on initial conditions. From these properties, we obtain the following result.
\begin{proposition}\label{propo1}
There exist $q_0 \in [0, \pi]$, $p_0$ such that for the solution $(q(t), p(t))$ of (\ref{eq1}) with the corresponding initial conditions $q(t_0) = q_0$, $p(t_0) = p_0$, the following holds $q(t) \in [0, \pi]$ for all $t > t_0$.
\end{proposition}
\begin{proof}
First, consider (\ref{eq1}) in the domain $G = \{ 0 < q < \pi \}$. Any solution leaving $G$ can be continued up to the boundary of $G$. At the same time, the solution cannot leave $G$ at the points where $q = 0$ and $p > 0$ or where $q = \pi$ and $p < 0$. Therefore, for any solution starting in $G $ there are three possibilities: it can never leave $G$; it can leave $G$ through the set $q = 0$, $p < 0$ or through the set $q = \pi$, $p > 0$; it can leave $G$ through the set $q = 0$, $p = 0$ or $q = \pi$, $p = 0$.

Let us now consider a continuous curve $p = \sigma(q)$, $t = t_0$, $0 \leqslant q \leqslant \pi$, where $\sigma$ is a continuous function and $\sigma(0) < 0$, $\sigma(\pi) > 0$. Consider all the solutions starting at this curve. Suppose that all these solutions leave $G$.

If some solution leaves $G$ through the set $q = 0$, $p < 0$ or $q = \pi$, $p > 0$, then all the solutions starting from close initial conditions also leave $G$ through close boundary points. It follows from the continuous dependence on the initial data.

Now consider the case when some solution, starting at the considered curve, reaches the line $q = 0$, $p = 0$ for the first time at moment $t = t^*$. This solution either stays in $q = 0$, $p = 0$ for all $t \geqslant t^*$ or leaves it at some $t = t^{**}$. If it stays in the line for all $t \geqslant t^*$, then it is the required solution. However, above we have supposed that all solutions leave $G$, i.e., our solution leaves line $q = 0$, $p = 0$ at $t = t^{**}$, where $t^{**} = t^* + \sup \{ \Delta t \geqslant 0 \colon q(t^* + t) = 0, \, p(t^* + t) = 0,\quad \forall\, 0 \leqslant t \leqslant \Delta t \}$. There are two possibilities: the pendulum can start moving either outside or inside the set $G$. If it moves inside $G$, then the map between the curve and the boundary $\partial G$ may become discontinuous because then there is a possibility for two close solutions to leave $G$ through the different components of the boundary ($q = 0$ and $q = \pi$). Below we prove that it is not the case.

For small $|q|$ and $p > 0$, we have $\dot p < 0$. Therefore, our solution can leave the line only to the set where $p \leqslant 0$, i.e., there exists $t^{***} > t^{**}$ such that $p(t) \leqslant 0$, for all $t \in [t^{**}, t^{***}]$. Moreover, for some $t \in [t^{**}, t^{***}]$ we have $p(t) < 0$ (if it is not true, then the solution do not leave the line at $t = t^{**}$). Since $\dot q = p$, we obtain that our solution, and solutions close to it, leave $G \cup \partial G$. Similarly, one can prove that if some solution reaches the line $p = 0$, $q = \pi$, it either stay in it forever or leaves $G \cup \partial G$.

Consider the continuous map from $\partial G$ to the points $q = 0$, $p = \sigma(0)$ and $q = \pi$, $p = \sigma(\pi)$ that maps two connected components of the boundary to these two points, respectively (to be more precise, it maps the plane $q = 0$ into the point $q = 0$, $p = \sigma(0)$, $t = 0$ and the plane $q = \pi$ into the point $q = \pi$, $p = \sigma(\pi)$ $t = 0$). We supposed that all solutions starting at the curve leave $G \cup \partial G$. Previously, we have shown that in this case there exists a continuous map between the curve $\sigma$ and the set $\partial G$. This map This map is a correspondence that assigns to every point of $\sigma$ the point of the first exit of the corresponding solution from $G$. Now we can consider the composition of the above two continuous maps. Finally, we obtain a continuous map between the curve and its boundary points. This contradiction proves the proposition.

\end{proof}
Note that from the proof it also follows that there exist infinitely many solutions without falling. Indeed, a one-parameter family of such solutions can be obtained if we consider a family of non-intersecting curves $\sigma(q)$.

Similarly, one can prove the following result, which contains sufficient conditions for the existence of a solution staying in $(0, \pi)$ for all $t \geqslant t_0$.
\begin{proposition}
Suppose that $\mu|\ddot \xi| < g$ for all $t \geqslant t_0$. Then there exist $q_0 \in (0, \pi)$, $p_0$ such that for the solution $(q(t), p(t))$ of (1) with the corresponding initial conditions $q(t_0) = q_0$, $p(t_0) = p_0$, the following holds $q(t) \in (0, \pi)$ for all $t > t_0$.
\end{proposition}

\begin{proof}
The proof is analogous to the previous one. The only difference is that it is possible to show that solutions starting in $G$ cannot leave this set at the points where $q = 0$ or $q = \pi$ and $p = 0$. Indeed, suppose that for the solution $(q(t), p(t))$, for some $t^*$, we have $q(t^*) = 0$ and $p(t^*) = 0$. Since for all $t$, we have  $-g/2l - \mu|\ddot \xi|/2l \leqslant  -g/2l + \mu|\ddot \xi|/2l < 0$, then any solution which reaches the plane $p=0$ at the point $q = 0$, leaves this plane. Moreover, we can conclude that this solution can reach the plane only from the region where $p > 0$. Taking into account the above inequalities, for small $\Delta t < 0$, we have
$$
q(t^* + \Delta t) = -\frac{g}{2l}\Delta t^2 - \frac{\mu}{2l}|\ddot \xi|\Delta t^2 + o(\Delta t^2) < 0.
$$
We obtain that the considered solution reaches the point $q = 0$ and $p = 0$ from the outside of $G$. Similarly, one can prove that solutions cannot leave $G$ at the point $q = \pi$, $p = 0$. 
\end{proof}

\section{Conclusion}
The presented proof is based on the topological ideas of so-called Wa\.zewski method (see \cite{wazewski1948principe}, \cite{reissig1963qualitative}), which can be also applied to other pendulum-like systems. For instance, in \cite{polekhin2014examples}, it was proved that, if the motion of the pendulum is frictionless, then for any $\ddot \xi$, there exists a solution without falling.  Note, that this result agrees with Proposition \ref{propo1} from this paper if we put $\mu = 0$. A good overview of the attempts to prove the existence of a solution that always remains above the horizontal line in the system without friction can be found in \cite{Srzednicki2017}. However, the system with dry friction is qualitatively different comparing to the frictionless case: in the latter case, there are no equilibrium points when $|\ddot\xi|\ne 0$. At the same time, if $\mu \ne 0$ and $|\ddot \xi|$ is relatively small, then there is a set of  equilibrium points in a vicinity of $q = \pi/2$. These points can be considered as solutions without falling, yet from the proof it can be seen that there also exists at least one non-constant solution that never falls.

The ideas that are used in the presented paper can also be used for systems with viscous friction and for more complex systems where the massive point is moving on a two-dimensional surface. Further development of the ideas of Wa\.zewski method can also be used to prove the existence of periodic solutions without falling in pendulum-like and general mechanical systems \cite{polekhin2015forced}, \cite{polekhin2016forced}, \cite{bolotin2015calculus}. Similar methods have been found useful in studying global controllability of an inverted pendulum \cite{Polekhin2017gsi, Polekhin2018}.

In conclusion, we would like to note that the above results hold for a wider class of friction models. In particular, it is possible to consider various sufficiently smooth Stribeck curves. Therefore, we may expect the existence of a falling free motion (however, possibly unstable) in a real system of an inverted pendulum with horizontally moving pivot point.  


\medskip
Received xxxx 20xx; revised xxxx 20xx.
\medskip

\end{document}